\documentclass[preprint,11pt]{elsarticle}

\usepackage{amssymb}
\usepackage{amsmath,amsthm,amsfonts,amssymb,latexsym,mathrsfs,color,hyperref}

\newtheorem{theorem}{Theorem}
\newtheorem{corollary}[theorem]{Corollary}
\newtheorem{proposition}[theorem]{Proposition}

\newtheorem{lemma}[theorem]{Lemma}

\newcommand{\negg}{{\rm neg\,}}

\newcommand{\cyc}{{\rm cyc\,}}

\newcommand{\excc}{{\rm exc\,}}

\newcommand{\msn}{\mathcal{S}_n}
\newcommand{\ms}{\mathfrak{S}}

\newcommand{\lrf}[1]{\lfloor #1\rfloor}

\newcommand{\mbn}{{\mathcal S}^B_n}

\journal{}

\begin{document}

\begin{frontmatter}



\title{The ratio monotonicity of Eulerian-type polynomials}


\author[focal]{Jun-Ying Liu}
\ead{jyliu6@163.com}
\author[focal]{Guanwu Liu}
\ead{liuguanwu@hotmail.com}
\author[focal]{Shi-Mei Ma}
\ead{shimeimapapers@163.com}
\cortext[cor1]{Corresponding author: Shi-Mei~Ma}
\author[focal]{Zhi-Hong Zhang}
\ead{zhangzhihongpapers@163.com}
\address[focal]{School of Mathematics and Statistics, Shandong University of Technology, Zibo, Shandong 255000, P.R. China}
\begin{abstract}
This paper is motivated by determining the location of modes of some unimodal Eulerian-type polynomials.
The notion of ratio monotonicity was introduced by Chen-Xia when they investigated the $q$-derangement numbers. 
Let $(f_n(x))_{n\geqslant 0}$ be a sequence of real polynomials satisfying the Eulerian-type recurrence relation
$$f_{n+1}(x)=(anx+bx+c)f_n(x)+ax(1-x)\frac{\mathrm{d}}{\mathrm{d}x}f_n(x),~f_0(x)=1,$$
where $a,b$ and $c$ are nonnegative integers. Assume that $\deg f_n(x)=n$.
Setting $g_n(x)=x^nf_n\left(\frac{1}{x}\right)$, we have 
$$g_{n+1}(x)=(anx+b+cx)g_n(x)+ax(1-x)\frac{\mathrm{d}}{\mathrm{d}x}g_n(x),$$
We find that if $a+c\geqslant b\geqslant c>0$, then $f_n(x)$ is bi-gamma-positive and 
$g_n(x)$ is ratio monotone.
As applications, we discover the ratio monotonicity of several Eulerian-type polynomials, including the $(\excc,\cyc)$ $q$-Eulerian polynomials, the $1/k$-Eulerian polynomials, a kind of generalized Eulerian polynomials studied
by Carlitz-Scoville, the $(\operatorname{des}_B,\negg)$ $q$-Eulerian polynomials over the hyperoctahedral group and the $r$-colored Eulerian polynomials. In particular, let $A_n(x,q)$ be the $(\excc,\cyc)$ $q$-Eulerian polynomials, we find that the polynomials $x^{n-1}A_n(1/x,q)$ are ratio monotone when $0<q\leqslant 1$, 
while $A_n(x,q)$ are ratio monotone when $1\leqslant q\leqslant 2$.
\end{abstract}


\begin{keyword}
Ratio monotonicity \sep  Eulerian polynomials \sep Recurrence systems\sep Unimodality
\MSC[2010] 26D05\sep 05A15
\end{keyword}
\end{frontmatter}

\section{Introduction}
The types $A$ and $B$ Eulerian polynomials can be respectively defined by
\begin{equation*}
\begin{split}
A_{n+1}(x)&=(nx+1)A_{n}(x)+x(1-x)\frac{\mathrm{d}}{\mathrm{d}x}A_{n}(x),\\
B_{n+1}(x)&=(2nx+x+1)B_{n}(x)+2x(1-x)\frac{\mathrm{d}}{\mathrm{d}x}B_{n}(x),
\end{split}
\end{equation*}
with $A_0(x)=B_0(x)=1$ (see~\cite{Brenti94,Chow08}). There has been much recent work on Eulerian-type recursions, see~\cite{Barbero14,Hwang20,Ma26,Zhu} for instance.
For example, Hwang-Chern-Duh~\cite{Hwang20} considered the general recurrence relation:
\begin{equation*}\label{Eulerian02}
\mathcal{P}_{n}(x)=(\alpha(x)n+\gamma(x))\mathcal{P}_{n-1}(x)+\beta(x)(1-x)\frac{\mathrm{d}}{\mathrm{d}x}\mathcal{P}_{n-1}(x),
\end{equation*}
where $\mathcal{P}_0(x),\alpha(x),\beta(x)$ and $\gamma(x)$ are given functions of $x$. They studied the limiting distribution of the coefficients of $\mathcal{P}_n(x)$. Very recently, Liu-Yan~\cite{Liu25} considered interlacing properties related to
the polynomials $f_n(x)$ satisfying the recurrence relation
$$f_n(x)=(a_1n+a_2+(b_1n+b_2)x)f_{n-1}(x)+c_1(n-1)xf_{n-2}(x)+x(a_3+b_3x)\frac{\mathrm{d}}{\mathrm{d}x}f_{n-1}(x),$$
with $f_0(x)=1$. Unimodal and log-concave polynomials arise often in combinatorics and other branches in mathematics, 
but to determine the location of modes may be a difficult challenge, see~\cite{Branden15,Brenti940} for surveys on this topic.
This paper is motivated by determining the location of modes of some unimodal Eulerian-type polynomials.

Let $a(x)=\sum_{i=0}^na_ix^i$ be a polynomial with nonnegative coefficients.
We say that $a(x)$ is {\it unimodal} if $a_0\leqslant a_1\leqslant \cdots\leqslant a_{k-1}\leqslant a_k\geqslant a_{k+1}\geqslant\cdots \geqslant a_n$ for some $k$, where the index $k$ is called the {\it mode} of $a(x)$. It is {\it log-concave} if $a_i^2\geqslant a_{i-1}a_{i+1}$ for all $1\leqslant i\leqslant n-1$.
Clearly, a log-concave sequence with no internal zeros is unimodal, but the converse is not true.
Darroch~\cite{Darroch64} showed that if $a(x)$ has 
only real nonpositive zeros, then $a(x)$ is unimodal with at most two modes and each mode $k$ satisfies 
$$\left\lfloor \frac{a'(1)}{a(1)} \right\rfloor\leqslant k\leqslant \left\lceil \frac{a'(1)}{a(1)} \right\rceil.$$
The polynomial $a(x)$ is called {\it gamma-positive} if there exist nonnegative numbers $\gamma_k$ such that
$$a(x)=\sum_{k=0}^{\lrf{{n}/{2}}}\gamma_kx^k(1+x)^{n-2k}.$$
Clearly, gamma-positivity is a property that polynomials with symmetric coefficients may have, which implies their unimodality and symmetry, see~\cite{Athanasiadis17} for a survey on this subject.

We say that $a(x)$ is {\it spiral} if
$a_n\leqslant a_0\leqslant a_{n-1}\leqslant a_1\leqslant \cdots\leqslant a_{\lrf{n/2}}$.
Following~\cite[Definition 2.9]{Schepers13}, the polynomial $a(x)$ is {\it alternatingly increasing} if
$a_0\leqslant a_n\leqslant a_1\leqslant a_{n-1}\leqslant\cdots \leqslant a_{\lrf{{(n+1)}/{2}}}$.
Clearly, if $a(x)$ is spiral and $\deg a(x)=n$, then $x^na(1/x)$ is alternatingly increasing, and vice versa.
The notion of ratio monotonicity for polynomials was introduced by Chen-Xia~\cite{Chen11}.
We say that $a(x)$ is {\it ratio monotone} if 
\begin{equation}\label{ratio-def1}
\frac{a_n}{a_0} \leqslant \frac{a_{n-1}}{a_1} \leqslant \cdots \leqslant \frac{a_{n-i}}{a_i}\leqslant \frac{a_{n-i-1}}{a_{i+1}} \leqslant \cdots \leqslant \frac{a_{n-\left[\frac{n-1}{2}\right]}}{a_{\left[\frac{n-1}{2}\right]}} \leqslant 1
\end{equation}
and
\begin{equation}\label{ratio-def2}
\frac{a_0}{a_{n-1}} \leqslant \frac{a_1}{a_{n-2}} \leqslant \cdots \leqslant 
\frac{a_{i-1}}{a_{n-i}} \leqslant\frac{a_{i}}{a_{n-i-1}} \leqslant \cdots \leqslant \frac{a_{\left[\frac{n}{2}\right]-1}}{a_{n-\left[\frac{n}{2}\right]}} \leqslant 1.
\end{equation}
Ratio monotonicity implies the log-concavity and spiral property.
From~\eqref{ratio-def1} and~\eqref{ratio-def2}, we have
$$\frac{a_{i+1}}{a_{i}}\leqslant \frac{a_{n-i-1}}{a_{n-i}} \leqslant \frac{a_i}{a_{i-1}}.$$
In~\cite{Chen09,Chen11}, Chen-Xia gave elegant proofs of the ratio monotonicity of Boros-Moll polynomials and $q$-derangement numbers.
Let $P(x)$ be a polynomial with nonnegative and nondecreasing coefficients.
Chen-Yang-Zhou~\cite{Chen10} proved the ratio monotone property of $P(x+1)$. Since then, there has been much attention on the 
ratio monotone property of enumerative polynomials. For example, Su-Sun~\cite{Su23} established ratio monotonicity of the coordinator
polynomials of the root lattice of type $B_n$.

Let $f(x)$ be a polynomial of degree $n$.
There is a unique symmetric decomposition $f(x)= a(x)+xb(x)$, where 
\begin{equation}\label{ax-bx-prop01}
a(x)=\frac{f(x)-x^{n+1}f(1/x)}{1-x},~b(x)=\frac{x^nf(1/x)-f(x)}{1-x}.
\end{equation}
Clearly, $a(x)$ and $b(x)$ are both symmetric. Moreover, $\deg a(x)=1+\deg b(x)$ or $b(x)=0$.
Following Br\"and\'{e}n-Solus~\cite{Branden18}, we call the ordered pair of polynomials $(a(x),b(x))$ the {\it symmetric decomposition} of $f(x)$.
According to~\cite[Definition~1.2]{Ma24}, if $a(x)$ and $b(x)$ are both gamma-positive, then $f(x)$ is said to be {\it bi-gamma-positive}.
Bi-gamma-positivity implies alternatingly increasing property and gamma-positivity can be seen as a degeneration of bi-gamma-positivity. 
For example, the polynomial $1+10x+4x^2$ is bi-gamma-positive, since 
$$1+10x+4x^2=(1+7x+x^2)+x(3+3x)=((1+x)^2+5x)+x(3(1+x)).$$

We can now present the main result of this paper.
\begin{theorem}\label{thmmain}
Let $(f_n(x))_{n\geqslant 0}$ be a sequence of real polynomials satisfying the Eulerian-type recurrence
\begin{equation}\label{fnx-recu}
f_{n+1}(x)=(anx+bx+c)f_n(x)+ax(1-x)\frac{\mathrm{d}}{\mathrm{d}x}f_n(x),
\end{equation}
with $f_0(x)=1$, where $a,b$ and $c$ are nonnegative integers. Suppose that $\deg f_n(x)=n$.
If $a+c\geqslant b\geqslant c>0$, then $f_n(x)$ is bi-gamma-positive and 
$x^nf_n(1/x)$ is ratio monotone. 
\end{theorem}

As Hwang-Chern-Duh~\cite[Section~4.1]{Hwang20} once said, one of the most common 
recursions with rich combinatorial properties among the extensions of Eulerian numbers is given as follows:
\begin{equation}\label{Eulerian5}
\mathcal{P}_{n+1}(x)=(qnx+(qr-p)x+p)\mathcal{P}_{n}(x)+qx(1-x)\frac{\mathrm{d}}{\mathrm{d}x}\mathcal{P}_{n}(x),~\mathcal{P}_{0}(x)=1,
\end{equation}
which covers more then 60 examples in OEIS (see~\cite[Section~4]{Hwang20}). The exponential generating function of $\mathcal{P}_{n}(x)$ has the closed-form~\cite[Eq.~(36)]{Hwang20}:
$$\sum_{n=0}^\infty \mathcal{P}_{n}(x) \frac{z^n}{n!}=\mathrm{e}^{p(1-x)z}\left(\frac{1-x}{1-x\mathrm{e}^{q(1-x)z}}\right)^r.$$
Using Theorem~\eqref{thmmain} and the recursions~\eqref{Eulerian5} and~\eqref{Qmx-recu02}, we get the following result.
\begin{corollary}
Let $\mathcal{P}_{n}(x)$ be the polynomials satisfying the recursion~\eqref{Eulerian5}.
If $q+2p\geqslant qr\geqslant 2p>0$, then $\mathcal{P}_{n}(x)$ is bi-gamma-positive and 
$x^n\mathcal{P}_{n}(1/x)$ is ratio monotone. If $q(1+r)\geqslant 2p\geqslant qr>0$, then $\mathcal{P}_{n}(x)$ is ratio monotone.
\end{corollary}
This paper is organized as follows. Section~\ref{Section02} is devoted to the proof of Theorem~\ref{thmmain}. 
In Section~\ref{Section03}, we present some applications of Theorem~\ref{thmmain}.
\section{Proof of Theorem~\ref{thmmain}}\label{Section02}
\subsection{A proof of the bi-gamma-positivity of $f_n(x)$}
Applying the formula~\eqref{ax-bx-prop01}, we have $f_n(x)=a_n(x)+xb_n(x)$, where
\begin{align*}
a_n(x)=\frac{f_n(x)-x^{n+1}f_n(1/x)}{1-x},~b_n(x)=\frac{x^nf_n(1/x)-f_n(x)}{1-x}.
\end{align*}
Note that
$x^n f_n(1/x)=x^n a_n(1/x)+x^{n-1} b_n(1/x)=a_n(x)+b_n(x)$.
Setting $g_n(x)=x^n f_n(1/x)$,
it is easy to verify that
\begin{equation}\label{Qmx-recu02}
g_{n+1}(x)=(anx+b+cx)g_n(x)+ax(1-x)\frac{\mathrm{d}}{\mathrm{d}x}g_n(x).
\end{equation}
Using~\eqref{fnx-recu} and~\eqref{Qmx-recu02}, $f_n(x)=a_n(x)+xb_n(x)$ and $g_n(x)=a_n(x)+b_n(x)$ can be rewritten as
\begin{align*}
a_{n+1}(x) + x b_{n+1}(x)
&=(anx+bx+c)a_n(x)+x(a(n-1)x+bx+a+c)b_n(x)+\\
&ax(1-x)\left(\frac{\mathrm{d}}{\mathrm{d}x} a_n(x)+x\frac{\mathrm{d}}{\mathrm{d}x} b_n(x)\right);
\end{align*}
\begin{align*}
a_{n+1}(x) + b_{n+1}(x)
&=(anx+b+cx)a_n(x)+(anx+b+cx)b_n(x)+\\
&ax(1-x)\left(\frac{\mathrm{d}}{\mathrm{d}x} a_n(x)+\frac{\mathrm{d}}{\mathrm{d}x} b_n(x)\right).
\end{align*}
In view of $$(1-x) a_{n+1}(x)=a_{n+1}(x) + x b_{n+1}(x)-x\left(a_{n+1}(x) + b_{n+1}(x)\right),$$
$$(x-1) b_{n+1}(x)=a_{n+1}(x) + x b_{n+1}(x)-\left(a_{n+1}(x) + b_{n+1}(x)\right),$$
we obtain the following recurrence system
\begin{equation}\label{recu-system}
\left\{
  \begin{array}{ll}
a_{n+1}(x)&=(anx+c+cx)a_n(x)+ax(1-x)\frac{\mathrm{d}}{\mathrm{d}x} a_n(x)+(a-b+c)xb_n(x),\\
b_{n+1}(x)&=(anx-ax+bx+b)b_n(x)+ax(1-x)\frac{\mathrm{d}}{\mathrm{d}x} b_n(x)+(b-c)a_n(x),
  \end{array}
\right.
\end{equation}
with the initial conditions $a_0(x)=1$ and $b_0(x)=0$. In particular, we have 
$$a_1(x)=c+cx,~b_1(x)=b-c;$$
$$a_2(x)=c^2+(ab-b^2+ac+2bc+c^2)x+c^2x^2,~b_2(x)=(b^2-c^2)(1+x).$$
Assume that 
$$a_n(x)=\sum_{k\geqslant 0}\alpha_{n,k}x^k(1+x)^{n-2k},~b_n(x)=\sum_{k\geqslant 0}\beta_{n,k}x^k(1+x)^{n-1-2k}.$$
Using~\eqref{recu-system}, it is routine to verify that 
\begin{equation}\label{recu-system02}
\left\{
  \begin{array}{ll}
\alpha_{n+1,k}=(ak+c)\alpha_{n,k}+2a(n-2k+2)\alpha_{n,k-1}+(a-b+c)\beta_{n,k-1},\\
\beta_{n+1,k}=(ak+b)\beta_{n,k}+2a(n-2k+1)\beta_{n,k-1}+(b-c)\alpha_{n,k},
  \end{array}
\right.
\end{equation}
with $\alpha_{0,0}=1,~\beta_{0,0}=0$ and $\alpha_{0,k}=\beta_{0,k}=0$ for $k\geqslant 1$. Therefore, when $a+c\geqslant b\geqslant c>0$, we see that $f_n(x)$ is bi-gamma-positive. Moreover, if $a>0$ and $b=c$, then $f_n(x)$ is reduced to a gamma-positive polynomial. \qed

Suppose that $p(x)$ and $q(x)$ have only real zeros, 
the zeros of $p(x)$ are $\xi_1\leqslant\cdots\leqslant\xi_n$,
and that those of $q(x)$ are $\theta_1\leqslant\cdots\leqslant\theta_m$.
Following~\cite{Wagner96}, we say that $p(x)$ {\it interlaces} $q(x)$ if $\deg q(x)=1+\deg p(x)$ and the zeros of
$p(x)$ and $q(x)$ satisfy
$\theta_1\leqslant\xi_1\leqslant\theta_2\leqslant\cdots\leqslant\xi_n
\leqslant\theta_{n+1}$.
We use the notation $p(x)\prec q(x)$ for $p(x)$ interlaces $q(x)$.
By~\cite[Theorem~2.1]{Liu07}, one can immediately get the following result.
\begin{proposition}
Let $\{f_n(x)\}_{n\geqslant 0}$ be a sequence of polynomials satisfying~\eqref{fnx-recu}. Then $f_n(x)$ has only real zeros and so it is unimodal.
\end{proposition}
\subsection{A proof of the ratio monotonicity of $f_n(x)$}
To prove the ratio monotonicity of $f_n(x)$, we need the following lemma.
\begin{lemma}\label{lemma2}
Suppose that $a_1,a_2,a_3,a_4,a_5,a_{6}$ are positive numbers satisfying
\begin{equation}\label{lemma1-equ}
\frac{a_1}{a_2}\leqslant \frac{a_3}{a_4}\leqslant\frac{a_5}{a_6},~\frac{a_2}{a_3}\leqslant\frac{a_4}{a_5},~a_3\leqslant a_5,~a_4\leqslant a_6,
\end{equation}
then we have 
$$\frac{\lambda_1a_1+(\lambda-\lambda_1)a_3}{\lambda_2a_2+(\lambda-\lambda_2)a_4}\leqslant 
\frac{\lambda_1a_3+(\lambda-\lambda_1)a_5+\mu(a_5-a_3)}{\lambda_2a_4+(\lambda-\lambda_2)a_6+\mu(a_6-a_4)},$$
where $0<\lambda_2\leqslant \lambda_1\leqslant \lambda$ and $\mu\geqslant 0$.
\end{lemma}
\begin{proof}
Note that 
$a_1(a_4+a_6)-a_2(a_3+a_5)=(a_1a_4-a_2a_3)+(a_1a_6-a_2a_5)\leqslant 0$.
We get 
\begin{equation}\label{e}
\frac{a_1}{a_2}\leqslant \frac{a_3+a_5}{a_4+a_6}.
\end{equation}

We now show that 
\begin{equation}\label{eq}
\frac{\lambda_1a_1+(\lambda-\lambda_1)a_3}{\lambda_2a_2+(\lambda-\lambda_2)a_4}\leqslant 
\frac{\lambda_1a_3+(\lambda-\lambda_1)a_5}{\lambda_2a_4+(\lambda-\lambda_2)a_6}.
\end{equation}
It follows from~\eqref{lemma1-equ} that 
$$a_1a_4\leqslant a_2a_3,~a_3a_6\leqslant a_4a_5,~a_1a_6\leqslant a_2a_5,~\frac{a_1}{a_4}= \frac{a_1}{a_2}\frac{a_2}{a_4}\leqslant\frac{a_5}{a_6}\frac{a_3}{a_5}=\frac{a_3}{a_6}.$$
Therefore, we obtain 
\begin{align*}
&\left(\lambda_1a_1+(\lambda-\lambda_1)a_3\right)\left(\lambda_2a_4+(\lambda-\lambda_2)a_6\right)-\left(\lambda_2a_2+(\lambda-\lambda_2)a_4\right)
\left(\lambda_1a_3+(\lambda-\lambda_1)a_5\right)\\
&=\lambda_1\lambda_2(a_1a_4-a_2a_3)+(\lambda-\lambda_1)(\lambda-\lambda_2)(a_3a_6-a_4a_5)+\\
&\lambda_1(\lambda-\lambda_2)a_1a_6-\lambda(\lambda_1-\lambda_2)a_3a_4-\lambda_2(\lambda-\lambda_1)a_2a_5\\
&=\lambda_1\lambda_2(a_1a_4-a_2a_3)+(\lambda-\lambda_1)(\lambda-\lambda_2)(a_3a_6-a_4a_5)+\\
&(\lambda_1\lambda-\lambda_2\lambda+\lambda_2\lambda-\lambda_1\lambda_2)a_1a_6-\lambda(\lambda_1-\lambda_2)a_3a_4-\lambda_2(\lambda-\lambda_1)a_2a_5\\
&=\lambda_1\lambda_2(a_1a_4-a_2a_3)+(\lambda-\lambda_1)(\lambda-\lambda_2)(a_3a_6-a_4a_5)+\\
&\lambda(\lambda_1-\lambda_2)(a_1a_6-a_3a_4)+\lambda_2(\lambda-\lambda_1)(a_1a_6-a_2a_5)\leqslant 0,
\end{align*}
which yields~\eqref{eq}, as desired. 

Note that 
\begin{align*}
&\left(\lambda_1a_3+(\lambda-\lambda_1)a_5\right)a_6-\left(\lambda_2a_4+(\lambda-\lambda_2)a_6\right)a_5\\
&=\lambda_1a_3a_6-\lambda_2a_4a_5+(\lambda_2-\lambda_1)a_5a_6\\
&=\lambda_1a_6(a_3-a_5)+\lambda_2a_5(a_6-a_4)\\
&\leqslant \lambda_1a_6(a_3-a_5)+\lambda_1a_5(a_6-a_4)\\
&=\lambda_1(a_3a_6-a_4a_5)\leqslant 0.
\end{align*}
So we obtain
\begin{equation}\label{eqq}
\frac{\lambda_1a_3+(\lambda-\lambda_1)a_5}{\lambda_2a_4+(\lambda-\lambda_2)a_6}\leqslant \frac{a_5}{a_6}\leqslant \frac{a_5-a_3}{a_6-a_4},
\end{equation}
where the last inequality can be easily verified.
In conclusion, we get
 \begin{equation*}\label{eqqq}
\frac{\lambda_1a_1+(\lambda-\lambda_1)a_3}{\lambda_2a_2+(\lambda-\lambda_2)a_4}\leqslant 
\frac{\lambda_1a_3+(\lambda-\lambda_1)a_5}{\lambda_2a_4+(\lambda-\lambda_2)a_6}\leqslant \frac{a_5-a_3}{a_6-a_4}.
\end{equation*}
Comparing with~\eqref{lemma1-equ} and~\eqref{e}, we get the desired result. This completes the proof.
\end{proof}

\noindent{\bf A proof of the ratio monotonicity of $f_n(x)$:}
\begin{proof}
Comparing~\eqref{fnx-recu} and~\eqref{Qmx-recu02}, we see that the following two statements are equivalent:
\begin{itemize}
  \item [$(i)$] If $a+c\geqslant b\geqslant c>0$, then $x^nf_n(1/x)$ is ratio monotone;
  \item [$(ii)$] If $a+b\geqslant c\geqslant b>0$, then $f_n(x)$ is ratio monotone.
\end{itemize}

Let $f_n(x)=\sum_{i=0}^{n}f_{n,i}x^i$.  
It follows from~\eqref{fnx-recu} that
\begin{equation}\label{recurrence relation}
f_{n+1, i}=(ai+c)f_{n, i} + (a(n-i) + a + b)f_{n, i-1}.
\end{equation} 
In the following discussion, assume that $a+b\geqslant c\geqslant b>0$.
Note that $$f_1(x)=c+bx,~f_2(x)=c^2+(ac+2bc+ab)x+b^2 x^2,$$
$$f_3(x)  = c^3 + (a^2 c+a^2b+3abc+3bc^2+3ac^2) x + (a^2 c+a^2b+3abc+3b^2c+3ab^2)x^2 + b^3 x^3 .$$
The result is true for $n=1,2,3$, since $b\leqslant c,~b^2\leqslant c^2,~c^2\leqslant ac + 2bc  + ab$ and
$$\frac{b^3}{c^3}\leqslant \frac{a^2 c+a^2b+3abc+3b^2c+3ab^2}{a^2 c+a^2b+3abc+3bc^2+3ac^2}\leqslant 1,~\frac{c^3}{a^2 c+a^2b+3abc+3b^2c+3ab^2}\leqslant 1,$$
where the last inequality can be derived by using the fact that $(a+b)^2\geqslant c^2$.

Suppose that $f_n(x)$ is the ratio monotone. 
When $n=2m$, we have
 \begin{equation}\label{ratio-fn-1}
  \frac{f_{2m, 2m}}{f_{2m, 0}} \leqslant \frac{f_{2m, 2m-1}}{f_{2m, 1}} \leqslant \cdots \leqslant \frac{f_{2m, 2m-i}}{f_{2m, i}} \leqslant \cdots \leqslant 
  \frac{f_{2m, m+1}}{f_{2m, m-1}}  \leqslant 1
\end{equation}  
and
\begin{equation}\label{ratio-fn-2}
  \frac{f_{2m, 0}}{f_{2m, 2m-1}} \leqslant \frac{f_{2m, 1}}{f_{2m, 2m-2}} \leqslant \cdots \leqslant \frac{f_{2m, i-1}}{f_{2m, 2m-i}} \leqslant \cdots \leqslant 
  \frac{f_{2m, m-1}}{f_{2m, m}} \leqslant 1.
\end{equation} 
We proceed by induction. In the following, we need to show that 
\begin{equation}\label{ratio-fn+1-1}
  \frac{f_{2m+1, 2m+1}}{f_{2m+1, 0}} \leqslant \frac{f_{2m+1, 2m}}{f_{2m+1, 1}}\leqslant \cdots \leqslant \frac{f_{2m+1, 2m+1-i}}{f_{2m+1, i}} \leqslant \cdots \leqslant 
  \frac{f_{2m+1, m+1}}{f_{2m+1, m}}  \leqslant 1
  \end{equation}  
and
\begin{equation}\label{ratio-fn+1-2}
  \frac{f_{2m+1, 0}}{f_{2m+1, 2m}} \leqslant \frac{f_{2m+1, 1}}{f_{2m+1, 2m-1}} \leqslant \cdots \leqslant \frac{f_{2m+1, i}}{f_{2m+1, 2m-i}} \leqslant \cdots \leqslant 
  \frac{f_{2m+1, m-1}}{f_{2m+1, m+1}} \leqslant 1.
\end{equation} 

We first establish~\eqref{ratio-fn+1-1}.
Note that
\begin{align*}
 &\Delta_1:= f_{2m+1, 2m+1} f_{2m+1, 1} - f_{2m+1, 0} f_{2m+1, 2m}\\
 &=bf_{2m,2m}((a+c)f_{2m,1}+(2ma+b)f_{2m,0})-cf_{2m,0}((2ma+c)f_{2m,2m}+(a+b)f_{2m,2m-1})\\
     =  & (b-c)(2am+b+c) f_{2m, 0} f_{2m, 2m} + b(a+c)f_{2m, 1} f_{2m, 2m} - c(a+b) f_{2m, 0} f_{2m, 2m-1}.
\end{align*}
From the left side of~\eqref{ratio-fn-1}, we see that $f_{2m, 1}f_{2m, 2m}  \leqslant f_{2m, 0} f_{2m, 2m-1}$.
It follows from $b\leqslant c$ that \begin{align*}
 &\Delta_1\leqslant (b-c)(2am+b+c) f_{2m, 0} f_{2m, 2m} + b(a+c)f_{2m, 0} f_{2m, 2m-1} - c(a+b) f_{2m, 0} f_{2m, 2m-1}\\
& = (b-c)(2am+b+c) f_{2m, 0} f_{2m, 2m} +a(b-c)f_{2m, 2m-1}\leqslant 0.
\end{align*}
Hence 
 $$\frac{f_{2m+1, 2m+1}}{f_{2m+1, 0}} \leqslant \frac{f_{2m+1, 2m}}{f_{2m+1, 1}}.$$
From the right sides of~\eqref{ratio-fn-1} and~\eqref{ratio-fn-2}, we see that $f_{2m,m+1}\leqslant f_{2m,m-1}\leqslant f_{2m,m}$.
So we have 
\begin{align*}
 &\Delta_2:= f_{2m+1, m+1} -f_{2m+1, m}\\
 &=(a(m+1)+c)f_{2m,m+1}+(ma+b)f_{2m,m}-(am+c)f_{2m,m}-(a(m+1)+b)f_{2m,m-1}\\
   &  =   (b-c)f_{2m, m}  +(a(m+1)+c)f_{2m,m+1}- (a(m+1)+b)f_{2m,m-1}\\
     &\leqslant  (b-c)f_{2m, m} +(a(m+1)+c)f_{2m,m-1}- (a(m+1)+b)f_{2m,m-1}\\
     &=(b-c)(f_{2m, m}-f_{2m,m-1})\leqslant 0,
\end{align*}
which yields that $f_{2m+1, m+1} \leqslant f_{2m+1, m}$.
We now ready to show that for $1\leqslant i\leqslant m-1$, we have
  $$\frac{f_{2m+1, 2m+1-i}}{f_{2m+1, i}} \leqslant  \frac{f_{2m+1, 2m-i}}{f_{2m+1, i+1}}.$$
From~\eqref{ratio-fn-1} and~\eqref{ratio-fn-2}, we observe that 
\begin{gather*}
  \frac{f_{2m, 2m-i+1}}{f_{2m, i-1}} \leqslant \frac{f_{2m, 2m-i}}{f_{2m, i}} \leqslant  \frac{f_{2m,2m-i-1}}{f_{2m,i+1}}\leqslant 1,\\
  \frac{f_{2m, i-1}}{f_{2m, 2m-i}} \leqslant \frac{f_{2m, i}}{f_{2m, 2m-i-1}} \leq 1.
\end{gather*}
In Lemma~\ref{lemma2}, setting $a_1=f_{2m, 2m-i+1},~a_2=f_{2m, i-1},~a_3=f_{2m, 2m-i},~a_4=f_{2m, i},~a_5=f_{2m,2m-i-1}$, $a_6=f_{2m,i+1}$,
$\lambda_1=a(2m+1-i)+c,~\lambda_2=a(2m+1-i)+b,~\lambda=(2m+1)a+b+c$ and $\mu =a$,
it follows from~\eqref{recurrence relation} that 
\begin{align*}
       & \frac{f_{2m+1, 2m+1-i}}{f_{2m+1, i}}=\frac{(a(2m+1-i)+c)f_{2m, 2m-i+1} + (ai+b)f_{2m, 2m-i}}{(a(2m+1-i)+b)f_{2m, i-1} + (ai+c)f_{2m, i}}\\
  \leqslant & \frac{(a(2m-i)+c)f_{2m, 2m-i} + (a(i+1)+b)f_{2m, 2m-i-1}}{(a(2m-i)+b)f_{2m, i} + (a(i+1)+c)f_{2m, i+1}}=\frac{f_{2m+1, 2m-i}}{f_{2m+1, i+1}}.
\end{align*}
Hence the proof of~\eqref{ratio-fn+1-1} is complete.

Next, we proceed to prove~\eqref{ratio-fn+1-2}.
From~\eqref{ratio-fn-1} and~\eqref{ratio-fn-2}, we see that
\begin{gather*}
  f_{2m, 0}f_{2m, 2m-2} \leqslant  f_{2m, 1}f_{2m, 2m-1},~f_{2m, 2m} \leqslant f_{2m, 2m-1} \leqslant  f_{2m, 2m-2},~f_{2m, 0} \leqslant f_{2m, 1}. 
\end{gather*}
Thus we get 
\begin{align*}
&\Delta_3:=f_{2m+1, 0} f_{2m+1, 2m-1} - f_{2m+1, 1} f_{2m+1, 2m}\\
 &= (c-a-b)(2am+b+c) f_{2m, 0} f_{2m, 2m-1} - (2am+c)(a+c)  f_{2m, 1} f_{2m, 2m} \\
     -  &   (2am+c)(2am+b)  f_{2m, 0} f_{2m, 2m} + c(2a+b)f_{2m, 0} f_{2m, 2m-2} - (a+c)(a+b) f_{2m, 1} f_{2m, 2m-1}\\
   \leqslant & (c-a-b)(2am+b+c) f_{2m, 0} f_{2m, 2m} - (2am+c)(a+c)  f_{2m, 1} f_{2m, 2m}  \\
     -  &  (2am+c)(2am+b)  f_{2m, 0} f_{2m, 2m} + a(c-a-b)f_{2m, 0} f_{2m, 2m-2}.
\end{align*}
Since $a+b\geqslant c$, we get $\Delta_3\leqslant 0$. Therefore, we obtain
$$\frac{f_{2m+1, 0}}{f_{2m+1, 2m}} \leqslant  \frac{f_{2m+1, 1}}{f_{2m+1, 2m-1}}.$$
From the right sides of~\eqref{ratio-fn-1} and~\eqref{ratio-fn-2}, we have
$f_{2m,m-2}\leqslant f_{2m,m+1}\leqslant f_{2m,m-1}\leqslant f_{2m,m}$.
So we find that 
\begin{align*}
&\Delta_4:=f_{2m+1, m-1} - f_{2m+1, m+1}\\
&=(am-a+c)f_{2m,m-1}+(am+2a+b)f_{2m,m-2}-(am+a+c)f_{2m,m+1}-(am+b)f_{2m,m}\\
&\leqslant (am-a+c)f_{2m,m}+(am+2a+b)f_{2m,m+1}-(am+a+c)f_{2m,m+1}-(am+b)f_{2m,m}\\
&=(c-a-b)(f_{2m,m}-f_{2m,m+1})\leqslant 0,
\end{align*}
which yields that $f_{2m+1, m-1}\leqslant f_{2m+1, m+1}$.
For $1\leqslant i\leqslant m-2$, we now show that 
$$\frac{f_{2m+1, i}}{f_{2m+1, 2m-i}} \leqslant  \frac{f_{2m+1, i+1}}{f_{2m+1, 2m-i-1}}.$$
From~\eqref{ratio-fn-1} and~\eqref{ratio-fn-2}, we have 
\begin{gather*}
  \frac{f_{2m, i-1}}{f_{2m, 2m-i}} \leqslant \frac{f_{2m, i}}{f_{2m, 2m-i-1}} \leqslant  \frac{f_{2m, i+1}}{f_{2m, 2m-i-2}}, \\
  \frac{f_{2m, 2m-i}}{f_{2m, i}} \leqslant \frac{f_{2m, 2m-i-1}}{f_{2m, i+1}}. 
\end{gather*}
In Lemma~\ref{lemma2}, setting $a_1=f_{2m, i-1},~a_2=f_{2m, 2m-i},~a_3=f_{2m, i}$, $a_4=f_{2m, 2m-i-1}$, $a_5=f_{2m, i+1}$, $a_6=f_{2m, 2m-i-2}$,
$\lambda_1=a(2m-i)+a+b$, $\lambda_2=a(2m-i)+c$, $\lambda=(2m+1)a+b+c$ and $\mu = a$,
it follows from~\eqref{recurrence relation} that 
\begin{align*}
&\frac{f_{2m+1, i}}{f_{2m+1, 2m-i}}=\frac{(2ma-ai+a+b)f_{2m,i-1}+(ai+c)f_{2m,i}}{(2ma-ai+c)f_{2m,2m-i}+(ai+a+b)f_{2m,2m-i-1}}\\
&\leqslant \frac{(2ma-ai+b)f_{2m,i}+(ai+a+c)f_{2m,i+1}}{(2ma-ai-a+c)f_{2m,2m-i-1}+(ai+2a+b)f_{2m,2m-i-2}}= \frac{f_{2m+1, i+1}}{f_{2m+1, 2m-i-1}},
\end{align*}
as desired. This completes the proof of~\eqref{ratio-fn+1-2}. 
The case that $n=2m+1$ can be dealt with in the same manner, and we omit it for simplicity.
\end{proof}
\section{Applications of Theorem~\ref{thmmain}}\label{Section03}
In this section, we apply Theorem~\ref{thmmain} to derive certain new results in a unified manner.
\subsection{$q$-Eulerian polynomials, $1/k$-Eulerian polynomials and generalized Eulerian polynomials}
Let $\msn$ be the set of all permutations of $[n]=\{1,2,\ldots,n\}$. 
For $\pi\in\msn$, we say that $i$ is an~{\it excedance} if $\pi(i)>i$.
Let $\excc(\pi)$ and $\cyc(\pi)$ be the numbers of excedances and cycles of $\pi$, respectively.
In~\cite{Brenti00}, Brenti studied the following $q$-Eulerian polynomials:
$$A_n(x,q)=\sum_{\pi\in\msn}x^{\excc(\pi)}q^{\cyc(\pi)}.$$
In particular, $A_1(x,q)=q,~A_2(x,q)=q(q+x)$ and $A_3(x,q)=q(q^2+(3q+1)x+x^2)$.
According to~\cite[Propositions~7.2,~7.3]{Brenti00}, the $q$-Eulerian polynomials $A_n(x,q)$ satisfy the recursion
\begin{equation}\label{A01}
A_{n+2}(x,q)=(nx+x+q)A_{n+1}(x,q)+x(1-x)\frac{\partial}{\partial x}A_{n+1}(x,q),
\end{equation}
and the exponential generating function of these polynomials is given as follows:
$$1+\sum_{n\geqslant 1}A_n(x,q)\frac{z^n}{n!}={\left(\frac{1-x}{\mathrm{e}^{z(x-1)}-x}\right)}^q.$$
When $q$ is a positive rational number,
Brenti showed that $A_n(x,q)$ has only real nonpositive simple zeros, and so it is log-concave and unimodal (\cite[Theorem~7.5]{Brenti00}).

Setting $L_{n}(x,q)=x^{n}A_{n+1}(1/x,q)$, it follows from~\eqref{Qmx-recu02} and~\eqref{A01} that 
\begin{equation}\label{gn01}
L_{n+1}(x,q)=(nx+qx+1)L_n(x,q)+x(1-x)\frac{\mathrm{d}}{\mathrm{d}x}L_n(x,q).
\end{equation}
Following Hwang-Chern-Duh~\cite[p.~26]{Hwang20}, the polynomials $L_n(x,q)$ can be called {\it LI Shanlan polynomials}, since these polynomials 
first appeared in his 1867 book.
Combining the recursions~\eqref{A01},~\eqref{gn01} and Theorem~\ref{thmmain}, we can give the following result.
\begin{corollary}\label{cor-Anq}
For any $n\geqslant 1$, we have the following results:
\begin{itemize}
  \item [\rm ($c_1$)] When $0<q\leqslant 1$, the polynomial $A_n(x,q)$ is bi-gamma-positive;
 \item [\rm ($c_2$)] When $0< q\leqslant 1$, the polynomial $x^{n-1}A_n(1/x,q)$ is ratio monotone;
\item [\rm ($c_3$)] When $1\leqslant q\leqslant 2$, the polynomial $A_n(x,q)$ is ratio monotone and $A_n(x,q)$ can be written as a sum of two gamma-positive polynomials with their degrees differing by 1.
\end{itemize}
\end{corollary}
It should be noted that $A_n(x,2)$ is the big descent polynomials over $\ms_{n+1}$, where a big descent is an index $i$ such that $\pi(i)\geqslant\pi(i+1)+2$, see~\cite[A120434]{Sloane}. We list the first few polynomials:
$$A_2(x,2)=4+2x,~A_3(x,2)=8+14x+2x^2,~A_4(x,2)=16 + 66 x + 36 x^2 + 2 x^3.$$
When $n\geqslant 4$, by~\eqref{recu-system} and~\eqref{gn01}, it is routine to verify that if $q>2$, then $A_n(x,q)$ can not be written as a sum of two gamma-positive polynomials with their degrees differing by 1. For examples, 
$$A_4(x,3)=81 + 201 x + 75 x^2 + 3 x^3,~A_4(x,4)=256 + 452 x + 128 x^2 + 4 x^3.$$

Following Savage-Viswanathan~\cite{Savage12}, the {\it $1/k$-Eulerian polynomials} $A_{n}^{(k)}(x)$ are defined by
\begin{equation*}\label{Ankx-def01}
\sum_{n=0}^\infty A_{n}^{(k)}(x)\frac{z^n}{n!}=\left(\frac{1-x}{\mathrm{e}^{kz(x-1)}-x} \right)^{\frac{1}{k}},
\end{equation*}
where $k\geqslant 1$.
They found that $A_{n}^{(k)}(x)$ are the ascent polynomials over $\rm k$-inversion sequences.
A more well known interpretation is given as follows (see~\cite{Savage12,Savage15}):
\begin{equation}\label{Ankx-def02}
A_{n}^{(k)}(x)=\sum_{\pi\in\msn}x^{\excc(\pi)}k^{n-\cyc(\pi)}.
\end{equation}
The polynomials $A_{n}^{(k)}(x)$ are also the ascent-plateau polynomials over $k$-Stirling permutations~\cite{Ma15,Ma26}. 
They satisfy the recursion
$$A_{n+2}^{(k)}(x)=(nkx+kx+1)A_{n+1}^{(k)}(x)+kx(1-x)\frac{\mathrm{d}}{\mathrm{d}x}A_{n+1}^{(k)}(x).$$
Below are these polynomials for $n\leqslant 3$:
$$A_1^{(k)}(x)=1,~A_2^{(k)}(x)=1+kx,~A_3^{(k)}(x)=1+3kx+k^2x(1+x).$$

The bi-gamma-positivity of $A_n^{(k)}(x)$ was first established in~\cite[Section~3.4]{Ma24}, and
Yan-Yang-Lin~\cite{Yan26} gave a nice combinatorial proof of this result. Combining Corollary~\ref{cor-Anq} and~\eqref{Ankx-def02}, we get the following result.
\begin{theorem}
The reciprocal $1/k$-Eulerian polynomials $x^{n}A_{n+1}^{(k)}(1/x)$ are ratio monotone.
\end{theorem}
For example, 
$A_5^{(k)}(x)=1 + (10 k  + 10 k^2  + 5 k^3  + k^4) x + (25 k^2  + 30 k^3  + 
 11 k^4) x^2 + (15 k^3  + 11 k^4) x^3 + k^4 x^4$.
When $k\geqslant 1$, we have 
$$\frac{1}{k^4}\leqslant \frac{10 k  + 10 k^2  + 5 k^3  + k^4}{15 k^3  + 11 k^4}\leqslant 1,$$
$$\frac{k^4}{10 k  + 10 k^2  + 5 k^3  + k^4}\leqslant \frac{15 k^3  + 11 k^4}{25 k^2  + 30 k^3  + 
 11 k^4}\leqslant 1. $$
 
Consider a kind of generalized Eulerian polynomials defined by
\begin{equation}\label{P01}
P_{n+1}(x;p,q)=(nx+px+q)P_{n}(x;p,q)+x(1-x)\frac{\partial}{\partial x}P_{n}(x;p,q),~P_0(x;p,q)=1.
\end{equation}
These polynomials were introduced by Morisita~\cite{Morisita71}, and they were also independently studied
by Carlitz-Scoville~\cite{Carlitz74}. It should be noted that these polynomials appeared in the context of random staircase tableaux, and Hitczenko-Janson~\cite{Hitczenko14}
investigated their asymptotic distribution.
Combining~\eqref{P01} and Theorem~\ref{thmmain}, we end this subsection by giving the following result.
\begin{corollary}
If $1+q\geqslant p\geqslant q>0$, then $P_n(x;p,q)$ is bi-gamma-positive and 
$x^nP_n(1/x;p,q)$ is ratio monotone. If $1+p\geqslant q\geqslant p>0$, then $P_n(x;p,q)$ is ratio monotone.
\end{corollary}
 
\subsection{The $q$-Eulerian polynomials of type $B$}
Let $\mbn$ denote the hyperoctahedral group of rank $n$. 
Elements of $\mbn$ are signed permutations of the set $\pm[n]=[n]\cup\{\overline{1},\ldots,\overline{n}\}$
with the property that $\sigma(\overline{i})=-\sigma(i)$ for all $i\in [n]$, where $\overline{i}=-i$.
Following Brenti~\cite{Brenti94},
the {\it $q$-Eulerian polynomials of type $B$} are defined by
$$B_n(x,q)=\sum_{\sigma\in\mbn}x^{\operatorname{des}_B(\sigma)}q^{\operatorname{neg}(\sigma)},$$
where $\operatorname{neg}(\sigma)=\#\{i\in [n]:~\pi(i)<0\}$ and $$\operatorname{des}_B(\pi)=\#\{i\in\{0,1,2,\ldots,n-1\}:~\pi(i)>\pi({i+1})~\&~\pi(0)=0\}.$$
They satisfy the recursion
$$B_{n+1}(x,q)=((1+q)nx+qx+1)B_n(x,q)+(1+q)x(1-x)\frac{\partial}{\partial x}B_{n}(x,q),$$
with the initial conditions $B_0(x,q)=1$ and $B_1(x,q)=1+qx$.
Setting $h_{n}(x)=x^{n}B_{n}(1/x,q)$, it follows from~\eqref{Qmx-recu02} that 
\begin{equation*}
h_{n+1}(x)=((1+q)nx+x+q)h_n(x,q)+(1+q)x(1-x)\frac{\partial}{\partial x}h_{n}(x,q).
\end{equation*}
Combining the above the recursions and Theorem~\ref{thmmain}, we get the following result.
\begin{corollary}
For any $n\geqslant 1$, we have the following results:
\begin{itemize}
  \item [\rm ($c_1$)] When $q\geqslant1$, the polynomial $B_n(x,q)$ is bi-gamma-positive;
 \item [\rm ($c_2$)] When $q\geqslant 1$, the polynomial $x^{n}B_n(1/x,q)$ is ratio monotone;
\item [\rm ($c_3$)] When $0<q\leqslant 1$, the polynomial $B_n(x,q)$ is ratio monotone and it can be written as a sum of two gamma-positive polynomials with their degrees differing by 1.
\end{itemize}
\end{corollary}
For example, 
$B_4(x,q)=1 + (11 + 32 q + 24 q^2 + 8 q^3 + q^4) x + (11 + 56 q + 96 q^2 + 
    56 q^3 + 11 q^4) x^2 + (1 + 8 q + 24 q^2 + 32 q^3 + 11 q^4) x^3 + 
 q^4 x^4$. When $q\geqslant 1$, we have 
 $$\frac{1}{q^4}\leqslant \frac{11 + 32 q + 24 q^2 + 8 q^3 + q^4}{1 + 8 q + 24 q^2 + 32 q^3 + 11 q^4}\leqslant 1;$$
$$\frac{q^4}{11 + 32 q + 24 q^2 + 8 q^3 + q^4}\leqslant \frac{1 + 8 q + 24 q^2 + 32 q^3 + 11 q^4}{11 + 56 q + 96 q^2 + 
    56 q^3 + 11 q^4}\leqslant 1.$$
\subsection{The $r$-colored Eulerian polynomials}
Following Steingr\'imsson~\cite{Steingrimsson},
the {\it $r$-colored Eulerian polynomial} can be defined by
\begin{equation}\label{Anrx-recu}
A_{n+1,r}(x)=(rnx+(r-1)x+1)A_{n,r}(x)+rx(1-x)\frac{\mathrm{d}}{\mathrm{d}x}A_{n,r}(x),~A_{0,r}(x)=1.
\end{equation}
When $r=1$ and $r=2$, the polynomial $A_{n,r}(x)$ reduces to the types $A$ and $B$ Eulerian polynomials $A_n(x)$ and $B_n(x)$, respectively.
Very recently, there has been much work devoted to the bi-gamma-positivity of the $r$-colored Eulerian polynomials and their variations, 
see~\cite{Athanasiadis20,Han2021,Ma24} for details. In particular, it is now well known that $A_{n,r}(x)$ is bi-gamma-positive when $r>2$, see~\cite[Eq.~(21)]{Athanasiadis20} and~\cite[Theorem~7.5]{Ma24}.

Setting $e_{n,r}(x)=x^{n}A_{n,r}(1/x)$, it follows from~\eqref{Qmx-recu02} and~\eqref{Anrx-recu} that 
\begin{equation*}\label{Anrx-recu}
e_{n+1,r}(x)=(rnx+x+r-1)e_{n,r}(x)+rx(1-x)\frac{\mathrm{d}}{\mathrm{d}x}e_{n,r}(x),~e_{0,r}(x)=1.
\end{equation*}
Therefore, by Theorem~\ref{thmmain}, we arrive at the following result.
\begin{corollary}
For any $n\geqslant 1$, we have the following results:
\begin{itemize}
  \item [\rm ($c_1$)] When $r\geqslant2$, the reciprocal polynomial $x^nA_{n,r}(1/x)$ is ratio monotone;
   \item [\rm ($c_2$)] When $1\leqslant r\leqslant 2$, the polynomial $A_{n,r}(x)$ is ration monotone, and it can be written as a sum of two gamma-positive polynomials with their degrees differing by 1.
   \end{itemize}
\end{corollary}

\section*{Acknowledgements}
The second author is supported by the National Natural Science Foundation of China (No. 12401459) and 
the Natural Science Foundation of Shandong Province of China (ZR2024QA075).
The third author is supported by the National Natural Science Foundation of China (No. 12071063) and
Taishan Scholars Program of Shandong Province (No. tsqn202211146).


\begin{thebibliography}{36}
%
%
\bibitem{Athanasiadis17}
C.A. Athanasiadis, \textit{Gamma-positivity in combinatorics and geometry}, S\'em. Lothar. Combin., \textbf{77} (2018), Article B77i.

\bibitem{Athanasiadis20}
C.A. Athanasiadis, \textit{Binomial Eulerian polynomials for colored permutations}, J. Combin. Theory Ser. A, \textbf{173} (2020), 105214.


%
\bibitem{Barbero14}
G.J.F. Barbero, J. Salas, E.J.S. Villase\~{n}or, \textit{Bivariate generating functions for a class of linear recurrences: general structure}, J. Combin. Theory Ser. A, \textbf{125} (2014), 146--165.
%
%
%
%
%
%
%
\bibitem{Branden15}
P. Br\"and\'{e}n, \textit{Unimodality, log-concavity, real-rootedness and beyond},
Handbook of Enumerative Combinatorics, Discrete Math. Appl. (Boca Raton), CRC Press, Boca Raton, FL (2015), pp. 437--483.

\bibitem{Branden18}
P. Br\"and\'{e}n, L. Solus, \textit{Symmetric decompositions and real-rootedness}, Int Math. Res Notices,  \textbf{2021} (2021), 7764--7798.
%
\bibitem{Brenti940}
F. Brenti, \textit{Log-concave and unimodal sequences in algebra, combinatorics, and geometry: an update}. In Jerusalem combinatorics 93, volume 178 of Contemp. Math., pp. 71--89. Amer. Math. Soc., Providence, RI, 1994.


\bibitem{Brenti94}
F. Brenti, \textit{$q$-Eulerian polynomials arising from Coxeter groups}, European J. Combin.,
\textbf{15} (1994), 417--441.

\bibitem{Brenti00}
F. Brenti, \textit{A class of $q$-symmetric functions arising from plethysm},
J. Combin. Theory Ser. A, \textbf{91} (2000), 137--170.

\bibitem{Carlitz74}
L. Carlitz, R. Scoville, \textit{Generalized Eulerian numbers: combinatorial applications},
J. Reine Angew. Math., \textbf{265} (1974), 110--137.

\bibitem{Chen09}
W.Y.C. Chen, E.X.W. Xia, \textit{The ratio monotonicity of the Boros-Moll polynomials},
Math. Comput., \textbf{78} (2009), 2269--2282.


\bibitem{Chen10}
W.Y. C. Chen, A.L.B. Yang, E.L.F. Zhou, \textit{Ratio monotonicity of polynomials derived from nondecreasing sequences}, 
Electron. J. Combin., \textbf{17} (2010), N37.


\bibitem{Chen11}
W.Y.C. Chen, E.X.W. Xia, \textit{The ratio monotonicity of the $q$-derangement numbers},
Discrete Math., \textbf{311} (2011), 393--397.

\bibitem{Chow08}
C.-O. Chow, \textit{On certain combinatorial expansions of the Eulerian polynomials}, Adv. in Appl. Math., \textbf{41} (2008), 133--157.
%

\bibitem{Darroch64}
J.N. Darroch, \textit{On the distribution of the number of successes in independent trials}, Ann. Math. Statist., \textbf{35} (1964), 1317--1321.

%
%
%
%
%
%
%
%

\bibitem{Wagner96}
J. Garloff, D.G. Wagner, \textit{Hadamard products of stable polynomials are stable}, J. Math. Anal. Appl., \textbf{202} (1996), 797--809.

\bibitem{Han2021}
B. Han, \textit{Gamma-positivity of derangement polynomials and binomial Eulerian polynomials for colored permutations},
J. Combin. Theory Ser. A, \textbf{182} (2021), 105459.

\bibitem{Hitczenko14}
P. Hitczenko, S. Janson, \textit{Weighted random staircase tableaux}, Combin. Probab. Comput., \textbf{23} (6) (2014), 1114--1147.


\bibitem{Hwang20}
H.-K. Hwang, H.-H. Chern, G.-H. Duh, \textit{An asymptotic distribution theory for Eulerian
recurrences with applications}, Adv. in Appl. Math., \textbf{112} (2020), 101960.
%
%
%
\bibitem{Liu07}
L.L. Liu, Y. Wang, \textit{A unified approach to polynomial sequences with only real zeros}, Adv. in Appl. Math., \textbf{38} (2007), 542--560.

\bibitem{Liu25}
L.L. Liu, X. Yan, \textit{Some interlacing properties related to the Eulerian and derangement polynomials}, Adv. in Appl. Math., \textbf{162} (2025), 102776.

%
%
\bibitem{Ma15}
S.-M. Ma, T. Mansour, \textit{The $1/k$-Eulerian polynomials and $k$-Stirling permutations}, Discrete Math., \textbf{338} (2015), 1468--1472.

%

\bibitem{Ma24}
S.-M. Ma, J. Ma, J. Yeh, Y.-N. Yeh, \textit{Excedance-type polynomials, gamma-positivity and alternatingly increasing property}, European J. Combin., \textbf{118} (2024), 103869.

\bibitem{Ma26}
S.-M. Ma, H. Qi, J. Yeh, Y.-N. Yeh, \textit{Stirling permutation codes. II}, J. Combin. Theory Ser. A, \textbf{217} (2026), 106093.

\bibitem{Morisita71}
M. Morisita, \textit{Measuring of habitat value by the ``environmental density'' method}, G.P. Patil, E.C. Pielou, W.E. Waters (Eds.), Statistical Ecology, \textbf{1} (1971), 379--401.

\bibitem{Savage12}
C.D. Savage and G. Viswanathan, \textit{The $1/k$-Eulerian polynomials}, Electron J. Combin.,
\textbf{19} (2012), \#P9.

\bibitem{Savage15}
C.D. Savage and M. Visontai, \textit{The $s$-Eulerian polynomials have only real roots}, Trans. Amer. Math.
Soc., \textbf{367} (2015), 763--788.
%
\bibitem{Schepers13}
J. Schepers, L.V. Langenhoven, \textit{Unimodality questions for integrally closed lattice polytopes}, Ann.
Comb., \textbf{17}(3) (2013), 571--589.

\bibitem{Sloane}
N.J.A. Sloane, The On-Line Encyclopedia of Integer Sequences,
published electronically at https://oeis.org.

\bibitem{Steingrimsson}
E. Steingr\'imsson, \textit{Permutation statistics of indexed permutations}, European J. Combin.,
\textbf{15} (1994), 187--205.

\bibitem{Su23}
X.-T. Su, F.-B. Sun, \textit{Refined ratio monotonicity of the coordinator polynomials of the root lattice of type $B_n$}, Open Math., \textbf{21} (2023), 20220555.


%
%
%
%

\bibitem{Yan26}
S.H.F. Yan, X. Yang, Z. Lin, \textit{Combinatorics on bi-$\gamma$-positivity of $1/k$-Eulerian polynomials}, J. Combin. Theory Ser. A, \textbf{217} (2026), 106092.


\bibitem{Zhu}
B.-X.Zhu, \textit{A generalized Eulerian triangle from staircase tableaux and tree-like tableaux}, J. Combin. Theory Ser. A, \textbf{172} (2020), 105206.
%
\end{thebibliography}
\end{document}